\documentclass[a4, 12pt]{amsart}
\usepackage[mathscr]{eucal}
\usepackage{amssymb}
\usepackage{latexsym}
\usepackage{amsthm}
\theoremstyle{plain}

\numberwithin{equation}{section}

{}

{}
\newtheorem{theorem}{Theorem}[section]
\newtheorem{Cor}{Corollary}[section]

\newtheorem{lemma}[theorem]{Lemma}
\newtheorem{Rem}{Remark}[section]

\date{}
\address{Qing-Ming Cheng: Department of Mathematics,  Graduate School of Science and Engineering,  
Saga University, Saga 840-8502,  Japan.  e-mail: cheng@cc.saga-u.ac.jp}
\address{
Hongcang Yang:
Academy of Mathematics and Systematical Sciences,
CAS,  Beijing 100080, China.  e-mail: yanghc2@netease.com}

\title
[Universal bounds for eigenvalues]{Universal bounds for  eigenvalues  \\ of  a  buckling problem II }
\author[Qing-Ming Cheng and Hongcang Yang]{Qing-Ming Cheng* and Hongcang Yang**}
\begin{document}

\maketitle

\vspace*{8 true mm}

\begin{abstract}
In this paper, we investigate universal estimates for eigenvalues of  a buckling problem. 
For a bounded domain in a Euclidean space, we give a positive contribution for obtaining
a sharp universal inequality for eigenvalues of the buckling problem. 
 For a domain in the unit sphere, we give an important improvement on the
results of Wang and Xia \cite{wx}.
\end{abstract}


\footnotetext{* Research partially Supported by a Grant-in-Aid for
Scientific Research from JSPS.}
\footnotetext{** Research partially Supported by SF of CAS.}

\vspace{8mm}
\section{ Introduction}

Let $M$ be an $n$-dimensional complete Riemannian manifold
and $\Omega \subset M$ a bounded domain in $M$ with piecewise 
smooth boundary $\partial \Omega$. 
 A Dirichlet eigenvalue problem of  Laplacian is given by 
 \begin{equation}
\left\{
     \begin{array}{ll}
     \triangle u = -\lambda u,& \ \ {\rm in} \ \ \Omega ,\\
     u=0 , & \ \ {\rm on}  \ \ \partial \Omega, 
     \end{array}
     \right.
\end{equation}
which is also called {\it a fixed membrane problem}, where
$\Delta$ denotes the Laplacian on $M$.
The spectrum of this eigenvalue problem is  real and  discrete.
 
 The following  eigenvalue problem of a biharmonic operator is called {\it a buckling problem}:
\begin{equation}
\left \{\aligned \Delta^2u  =&-\Lambda \Delta u \quad in  \ \ \Omega ,\\
u |_{\partial \Omega}=&\left .\frac {\partial u}{\partial \nu}
\right|_{\partial \Omega} =0,
\endaligned \right. 
\end{equation}
which describes the critical buckling load of a 
clamped  plate subjected to a  uniform compressive force around its boundary,
where $\nu$ is the outward unit normal vector field of the boundary $\partial \Omega$.
It is known that the spectrum of the buckling problem is also real and discrete. 
 
 When $\Omega \subset \mathbf  R^n$ be a bounded domain in an 
 $n$-dimensional Euclidean space $\mathbf  R^n$,  Payne, P\'olya and Weinberger \cite{ppw1} and \cite{ppw2}
 proved the  following inequality  for  eigenvalues of the eigenvalue problem (1.1):  
 for $k=1, 2, \cdots$,
 \begin{equation}
 \lambda_{k+1}-\lambda_k
 \leq \frac4{kn}\sum_{i=1}^k\lambda_i.
 \end{equation}
One  calls it  a universal inequality  since  it  does not depend on the 
 domain $\Omega$.

On the other hand,  Payne, P\'olya and Weinberger \cite{ppw1} and \cite{ppw2} also studied
eigenvalues of the buckling problem on  a bounded domain $\Omega$ in $\mathbf {R}^n$
and intended to derive a universal 
inequality for eigenvalues of the buckling problem. But 
it is very hard to deal with  this problem.  They only proved, for $n=2$,
$$
\Lambda_2\leq 3\Lambda_1.
$$
As an open problem,  Payne, P\'olya and Weinberger \cite{ppw1} and \cite{ppw2} proposed the following:

\vskip 5mm
\noindent
{\bf Problem}. {\it 
Whether can one obtain a universal inequality for 
 eigenvalues of the buckling problem {\rm (1.2)} on a bounded domain in a Euclidean space, 
 which is similar   to the 
 universal inequality  {\rm (1.3)}  for the eigenvalues of the fixed membrane 
 problem }(1.1) ?

\vskip 5mm

\noindent
For lower order eigenvalues, Hile and Yeh \cite{hy} and so on improved the result of 
Payne, P\'olya and Weinberger to 
$$
\Lambda_2\leq \frac{n^2+8n+20}{(n+2)^2}\Lambda_1.
$$ 
Furthermore, Ashbaugh \cite{ash2} (cf. \cite{ash1}) has obtained
\begin{equation*}
\sum_{i=1}^n \Lambda_{i+1} \le (n+4 )\Lambda_1
\end{equation*}
and he has commented that {\it to obtain a universal inequality for eigenvalues of the
buckling problem remains a challenge for mathematicians since 1955}. 
Many mathematicians have intended to attack  
this problem, but it remains open for almost 50 years. 

As one know that in order to obtain  a universal inequality for  eigenvalues of the buckling problem,
it is a key to find appropriate trial functions.  Cheng and Yang \cite{cy2}, by introducing a new method 
to  construct  trial functions for the buckling problem,  have obtained 
the following  universal inequality for   eigenvalues  of the buckling problem (1.2):
\begin{equation}
\sum_{i=1}^k(\Lambda_{k+1}-\Lambda_i)^2\le
\frac {4(n+2)}{n^2}\sum_{i=1}^k
(\Lambda_{k+1}-\Lambda_i)\Lambda_i.
\end{equation}
 Thus, the  problem  proposed by  Payne, P\'olya and Weinberger has been solved affirmatively.
By making use of the asymptotic formula of Weyl for eigenvalues of the Dirichlet  eigenvalue problem of the Laplacian and  one of Agmon \cite{ag}  and Pleijel \cite{ple} for  eigenvalues of the clamped plate  problem,
we can have the asymptotic formula of eigenvalues for the buckling problem according to the variational 
characterization for eigenvalues of the buckling problem:
\begin{equation}
\Lambda_k\sim
\dfrac{4\pi^2}{(\omega_n\text{vol}\Omega)^{\frac2n}}k^{\frac2n}, \ \ k\to \infty,
\end{equation}
where $\omega_n$ denotes the volume of the unit ball in $\mathbf R^n$. By the results of Li and Yau \cite{ly} and 
the variational  characterization for eigenvalues, 
one can obtain a  lower bound for eigenvalues of the buckling problem (cf. Levine and Protter \cite{lp}):
\begin{equation}
\dfrac1k\sum_{j=1}^k\Lambda_j\geq
\dfrac{n}{n+2}\dfrac{4\pi^2}{(\omega_n\text{vol}\Omega)^{\frac2n}}k^{\frac2n}.
\end{equation}

On the other hand, by making use of the  recursion formula in \cite{cy3},  one can obtain an upper bound for eigenvalues of the buckling problem, which is sharp in the sense of the order of $k$, if one can get  a sharp universal inequality for eigenvalues of the buckling problem as the following (cf. \cite{cy2}): 
\vskip 5mm
\noindent
{\bf Conjecture}. {\it  Eigenvalues of the buckling problem on a bounded 
 domain in a Euclidean space $ \mathbf {R}^n$ satisfy
 the following universal inequality}:
$$
\sum_{i=1}^k(\Lambda_{k+1}-\Lambda_i)^2\le\frac 4n
\sum_{i=1}^k(\Lambda_{k+1}-\Lambda_i)\Lambda_i.
$$
\vskip 5mm

\noindent
Therefore, the next  landmark  goal for  the study on eigenvalues of the buckling problem will be to prove 
the above sharp universal inequality.

In \cite{cy2}, we decompose $x^p \nabla u_i$ into 
\begin{equation}
x^p \nabla u_i= \nabla h_{pi} + {\bf w}_{pi},
\end{equation}
where the notations used may be found  in section 2.  We make use of the function $h_{ip}$ 
to construct appropriate trial functions. In order to get our universal inequality, we 
 estimated $L^2$-norm of ${\bf w}_{pi}$ in \cite{cy2}. As one knows that to find
 new appropriate trial functions is very difficult,  many years were spent
 for constructing appropriate trial functions in \cite{cy2}. In this paper, we shall also use 
 the trial functions constructed in \cite{cy2} and  our main observation is to introduce  
 new functions $q_{pi}$ and a careful exploitation of  
 $\nabla q_{pi} =\nabla(x^p u_i- h_{pi})$ and $\Delta{\bf w}_{pi}$. 
 Furthermore, the estimate on lower bound of  $L^2$-norm of $\nabla q_{pi}$ will play 
 an important role in the proof of our theorem 1.1.
 In this paper, we will prove
  $$
\Lambda_i\sum_{p=1}^n\|\nabla q_{pi}\|^2\geq \dfrac53.
$$
 If one can prove that $L^2$-norm of $\nabla q_{pi}$ satisfies
\begin{equation}
\Lambda_i\sum_{p=1}^n\|\nabla q_{pi}\|^2\geq 3,
\end{equation}
then the sharp universal inequality  in the above conjecture will be obtained
(see the remark 2.1 in section 2).
In order  to prove the inequality (1.8), we have spent several years. But we still 
can not prove it. Hence, we  hope to share our new ideas with mathematicians
who are interested in this field such that  the landmark goal in the study on 
eigenvalues of the buckling problem will be realized finally,
which also is one of our main purposes to publish this paper.

\vskip 2mm
\noindent
\begin{theorem} Let $\Lambda_i$ be the $i$-th eigenvalue 
of the buckling problem {\rm (1.2)} for a bounded domain 
$\Omega \subset \mathbf {R}^n$. Then, we have
\begin{equation}
\sum_{i=1}^k(\Lambda_{k+1}-\Lambda_i)^2\le
\frac {4(n+\frac43)}{n^2}\sum_{i=1}^k
(\Lambda_{k+1}-\Lambda_i)\Lambda_i.
\end{equation}
\end{theorem}

\begin{Rem}
Since our universal inequality is a quadratic inequality of the eigenvalue $\Lambda_{k+1}$, we can conclude
an upper bound of the gap between  two consecutive eigenvalues as in \cite{cy2} from {\rm (1.9)}. We will not give  it  
in details.
\end{Rem}

When $M$ is an $n$-dimensional unit sphere $S^n(1)$, Wang and Xia \cite{wx}
have studied the buckling problem on a domain $\Omega$  in $S^n(1)$.
They have obtained a  universal inequality for eigenvalues of the 
buckling problem, namely, they have proved
 that eigenvalues 
of the buckling problem {\rm (1.2)} on a  domain 
$\Omega$ in the unit sphere $S^n(1)$ satisfy
\begin{equation}
\begin{aligned}
&2\sum_{i=1}^k(\Lambda_{k+1}-\Lambda_i)^2\\
&\le\sum_{i=1}^k
(\Lambda_{k+1}-\Lambda_i)^2\bigl\{\delta\Lambda_i
+\dfrac{\delta^2\bigl(\Lambda_i-(n-2)\bigl )}{4(\delta\Lambda_i+n-2)}\bigl\}\\
&+\dfrac1{\delta}\sum_{i=1}^k
(\Lambda_{k+1}-\Lambda_i)(\Lambda_i+\dfrac{(n-2)^2}4),
\end{aligned}
\end{equation}
where $\delta$ is an arbitrary positive constant.

According to our knowledge, we think that eigenvalues of the buckling problem
on a domain  in $S^n(1)$ should satisfy
\begin{equation*}
\sum_{i=1}^k(\Lambda_{k+1}-\Lambda_i)^2\le
\frac {4}{n}\sum_{i=1}^k
(\Lambda_{k+1}-\Lambda_i)(\Lambda_i+\dfrac{n^2}{4}).
\end{equation*}
Since one needs to use covariant derivative for the unit sphere, in order to exchange
the orders of covariant derivatives, one must use the Bochner formula, which is different from
the case of the Euclidean spaces. Thus, one needs to deal 
with the terms of Ricci curvature. Hence, it will be a very hard work to obtain the 
above universal inequality. The second purpose in this paper is to  give an important improvement for   
the result of Wang and Xia.
\vskip 2mm
\noindent
\begin{theorem}  Eigenvalues $\Lambda_i$'s of the buckling problem {\rm (1.2)} on a  domain 
$\Omega$ in the unit sphere $S^n(1)$ satisfy
\begin{equation}
\begin{aligned}
&2\sum_{i=1}^k(\Lambda_{k+1}-\Lambda_i)^2
+(n-2)\sum_{i=1}^k\dfrac{(\Lambda_{k+1}-\Lambda_i)^2}{\Lambda_i-(n-2)}\\
&\le\sum_{i=1}^k
(\Lambda_{k+1}-\Lambda_i)^2\bigl\{\Lambda_i-\dfrac{n-2}{\Lambda_i-(n-2)}\bigl \}\delta_i\\
&+\sum_{i=1}^k\dfrac{(\Lambda_{k+1}-\Lambda_i)}{\delta_i}(\Lambda_i+\dfrac{(n-2)^2}4)
\end{aligned}
\end{equation}
for an arbitrary positive non-increasing monotone sequence $\{\delta_i\}_{i=1}^k$. 
\end{theorem}

\begin{Rem} It is obvious that our result is sharper than one of 
Wang and Xia \cite{wx} even if we take $\delta_i=\delta$ for any $i$. 
Since our universal inequality is a quadratic inequality of $\Lambda_{k+1}$,
we can obtain an explicit upper bound for the eigenvalue $\Lambda_{k+1}$
from {\rm (1.11)}.
\end{Rem}
In particular, when $n=2$, we have
\begin{Cor} 
Eigenvalues $\Lambda_i$'s of the buckling problem {\rm (1.2)} on a  domain 
$\Omega$ in the unit sphere $S^2(1)$ satisfy
\begin{equation}
\begin{aligned}
&\sum_{i=1}^k(\Lambda_{k+1}-\Lambda_i)^2
\le\sum_{i=1}^k(\Lambda_{k+1}-\Lambda_i)\Lambda_i^2.
\end{aligned}
\end{equation}
\end{Cor}
\begin{proof} Since $n=2$, from the theorem 1.2 and taking $\delta_i=\dfrac1{\Lambda_i}$,  
for $i=1, 2, \cdots, k$, 
for which  $\{\delta_i\}_{i=1}^k$  is a positive non-increasing monotone sequence, 
we finish
the proof of the corollary 1.1.
\end{proof}

\begin{Rem}
About the recent developments in universal inequalities for eigenvalues of 
the Dirichlet eigenvalue problem of the Laplacian and the clamped plate problem,
readers can see \cite{ah}, \cite{cc},  \cite{cim}, \cite{cy1}, \cite{cy3}, \cite{ehi}, \cite{h} and \cite{wx2}. 
\end{Rem}

\vskip 3mm
\noindent
{\it Acknowledgements.} We would like to express our gratitude 
to the referee for  valuable comments and suggestions.

\vspace{8mm}
\section{Proof of the theorem 1.1}

For the convenience of readers, we review the method for constructing 
trial functions introduced by Cheng and Yang \cite{cy2}.  In this section,
$\Omega$ is assumed to be a bounded domain in $\mathbf{R}^n$.
For functions $f$ and $h$, we define {\it Dirichlet inner product}
$(f,h)_D$ of $f$ and $h$ by
$$
(f,h)_D=\int_{\Omega}\langle\nabla f, \nabla h\rangle.
$$
{\it  Dirichlet norm} of a function $f$ is defined by
$$
||f||_D=\{(f,f)_D\}^{1/2}=\left (\int_{\Omega}
\sum_{\alpha=1}^n|\nabla_{\alpha} f |^2\right )^{1/2}.
$$
Let $u_i$ be  the $i$-th orthonormal eigenfunction of the buckling  problem
(1.2) corresponding to the eigenvalue $\Lambda_i$,
namely,  $ u_i$ satisfies
\begin{equation}
\begin{cases}
&\Delta^2u_i  =-\Lambda_i \Delta u_i \quad in \ \ \Omega ,\\
& u_i |_{\partial \Omega}=\left .\frac {\partial u_i}{\partial \nu}
\right|_{\partial \Omega} =0\\
&(u_i,u_j)_D=\int_{\Omega} \langle\nabla u_i, \nabla u_j\rangle =\delta_{ij}.
\end{cases}
\end{equation}
 $H^2_2(\Omega)$  defined by 
$$
H^2_2 (\Omega)=\{f:  f ,\nabla_{\alpha} f,
\nabla_{\alpha}\nabla_{\beta}f
\in L^2 (\Omega) ,\quad \alpha, \beta=1,\dots,n\}
$$
is  a Hilbert space with norm $\| \cdot \|_2$:
$$
\|f\|_2=\left (\int_{\Omega} |f|^2+ \int_{\Omega}
|\nabla f |^2 +\sum_{\beta,\alpha=1}^n(\nabla_{\alpha}
\nabla_{\beta} f)^2\right )^{1/2}.
$$
Let $H^2_{2,D}(\Omega)$ be a subspace of
$H_2^2(\Omega)$defined as
$$
H^2_{2, D} (\Omega)=\left \{ f \in H^2_2(\Omega): \ 
f|_{\partial M}= \left . \frac {\partial}{\partial \nu} f\right |
_{\partial \Omega} =0
\right \}.
$$
The biharmonic
operator $\Delta ^2$ defines a self-adjoint operator
acting on $ H^2_{2, D}(\Omega) $ with discrete eigenvalues 
$\{ 0< \Lambda_1 \leq \Lambda_2 \le \dots\leq \Lambda_k\leq \cdots \}$ for the buckling problem (1.2)
and the eigenfunctions defined in (2.1)
$$
\{u_i\}_{i=1}^{\infty}= \{ u_1,  u_2, \cdots,  u_k, \cdots \}
$$
form a complete orthogonal basis for Hilbert space 
$ H^2_{2,D}(\Omega)$. 
We define 
an inner product  $( {\bf f}, {\bf h})$ for vector-valued functions  $ {\bf f} =(
f^1, f^2, \cdots,  f^n ) \in \mathbf  R^n$ and ${\bf h}=(h^1, h^2, \cdots,  h^n)
\in \mathbf R^n $ by
$$
( {\bf f},{\bf h}  )\equiv\int_{\Omega} \langle{\bf f},
{\bf h}\rangle
=\int_{\Omega} \sum_{\alpha=1}^n f^{\alpha}h^{\alpha}.
$$
The norm of ${\bf f}$ is defined by
$$
\|{\bf f } \|  =\left (
{\bf f }  ,{\bf f } \right )^{1/2}=
\left \{\int_{\Omega} \sum_{\alpha=1}^n (f^{\alpha} )^2
\right \}^{1/2}.
$$
Denote a Hilbert space ${\bf H}^2_1(\Omega)$  of 
the vector-valued functions as
$$
{\bf H}^2_1(\Omega)=\{
{\bf f}:f^{\alpha}, \nabla_{\beta}f^{\alpha } \in L^2(\Omega)
, \ \text{for} \ \ \alpha,\beta=1,\dots,n\}
$$
with norm $ \|\cdot \|_1$:
$$
\|{\bf f}\|_1=\left (\|{\bf f}\|^2+\int_{\Omega}\sum_{\alpha,\beta=1}^n
|\nabla_{\alpha} f^{\beta}|^2 \right )^{1/2}.
$$
Let ${\bf H}^2_{1,D}(\Omega) \subset {\bf H}^2_1(\Omega)$ be 
a subspace of ${\bf H}^2_1(\Omega)$  spanned by  
the vector-valued functions $\{ \nabla u_i\}_{i=1}^{\infty}$, which  
form  a complete orthonormal basis of ${\bf H}^2_{1,D}(\Omega) $.

It is easy to see that
for any $f \in H^2_{2,D}(\Omega)$,  
$\nabla f \in {\bf H}^2_{1,D}(\Omega)$
and for any ${\bf h}\in {\bf H}^2_{1,D}(\Omega)$, there exists  a function
$f \in H^2_{2,D}(\Omega )$ such  that
${\bf h }  = \nabla f$. 

Let $x^p$ for $p=1, 2, \cdots, n$ be the $p$-th coordinate function
of $\mathbf {R}^n$.
For  the vector-valued function  $x^p \nabla u_i, i=1, \dots, k$, we decompose
it into
\begin{equation}
x^p \nabla u_i= \nabla h_{pi} + {\bf w}_{pi},
\end{equation}
where  $h_{pi} \in H^2_{2,D}(\Omega)$ and $\nabla h_{pi}$ is the projection of $x^p\nabla u_i $
onto ${\bf H}^2_{1,D}(\Omega)$ and
${\bf w}_{pi} \perp H^2_{1,D} (\Omega)$. Thus, 

\begin{equation}
({\bf w}_{pi} , \nabla u )
=\int_{\Omega} \sum_{j=1}^n w_{pi}^{j}
\nabla_{j} u =0, \ 
\text{for any} \ u \in H^2_{2,D}(\Omega).
\end{equation}
Therefore, since $H^2_{2,D}(\Omega)$ is dense in $L^2(\Omega)$
and $C^1(\Omega)$ is dense in $L^2(\Omega)$,
we have, for any function $h\in C^1(\Omega)\cap L^2(\Omega)$,
\begin{equation}
( {\bf w}_{pi} , \nabla h )=0.
\end{equation}
Hence, from the definition of ${\bf w}_{pi}$ and (2.4), we have
\begin{equation}
\begin{cases}
&{\bf w}_{pi}|_{\partial \Omega}=0,\\
&\|\text { div}{\bf w}_{pi}\|^2=0,\quad (\text { div}{\bf w}_{pi} \equiv\sum_{j=1}^n
\nabla_{j} w_{pi}^{j} ).
\end{cases}
\end{equation}
We define function $\varphi_{pi}$ by
\begin{equation}
\varphi_{pi}=h_{pi}-\sum_{j=1}^k b_{pij}u_j,
\end{equation}
where 
$$
b_{pij}=\int x^p\langle\nabla u_i, \nabla u_j\rangle=b_{pji}.
$$ 
It is easy to check, from the definition (2.2) of $h_{pi}$,  that 
$\varphi_{pi}$ satisfies
\begin{equation}
\varphi_{pi}|_{\partial \Omega}
=\frac{\partial\varphi_{pi}}{\partial\nu}|_{\partial\Omega}=0\ \
\text{and} \ \ (\varphi_{pi},u_j)_D=(\nabla \varphi_{pi},\nabla u_j)=0,
\end{equation}
for any $j=1, 2, \cdots, k$.
Hence,  we know that $\varphi_{pi}$ is a trial  function.\par

In order to prove our theorem 1.1, we prepare three lemmas.
\begin{lemma} For any  $p$ and $i$, we have 
\begin{equation}
1+2\|\langle\nabla x^p,\nabla u_i\rangle\|^2
=2\int  x^pu_i\langle\nabla x^p, \nabla(\Delta u_i)\rangle.
\end{equation}
\end{lemma}
\begin{proof} From the Stokes' formula, we have
\begin{equation*}
\begin{aligned}
&\int \langle x^pu_i\nabla x^p, \nabla(\Delta u_i)\rangle\\
&=-\int\text{div}( x^pu_i\nabla x^p) \Delta u_i\\
&=-\int u_i\Delta u_i-\int x^p\Delta u_i\langle\nabla x^p,\nabla u_i\rangle,
\end{aligned}
\end{equation*}
\begin{equation*}
\begin{aligned}
&\int x^p\Delta u_i\langle\nabla x^p,\nabla u_i\rangle\\
&=-\int\langle \nabla x^p, \nabla u_i\rangle^2
-\int x^p\langle \nabla u_i, \nabla\langle \nabla x^p,\nabla u_i\rangle\rangle\\
&=-\|\langle \nabla x^p,\nabla u_i\rangle\|^2
+\int\text{div}(x^p\nabla\langle \nabla x^p,\nabla u_i\rangle)u_i\\
&=-\|\langle \nabla x^p,\nabla u_i\rangle\|^2
+\int\langle\nabla x^p,\nabla\langle \nabla x^p,\nabla u_i\rangle\rangle u_i
+\int x^pu_i\Delta\langle \nabla x^p,\nabla u_i\rangle\\
&=-\|\langle \nabla x^p,\nabla u_i\rangle\|^2
-\int\langle\nabla x^p,\nabla u_i\rangle^2
+\int x^pu_i\langle \nabla x^p,\nabla (\Delta u_i)\rangle.\\
\end{aligned}
\end{equation*}
Since $\|\nabla u_i\|^2=1$, we have 
\begin{equation*}
1+2\|\langle\nabla x^p,\nabla u_i\rangle\|^2
=2\int  x^pu_i\langle\nabla x^p, \nabla(\Delta u_i)\rangle.
\end{equation*}
\end{proof}
According to 
$x^p\nabla u_i=\nabla {h}_{pi}+{\bf w}_{pi}$ 
and $\nabla(x^pu_i)\in H_{1,D}^2(\Omega)$, we have 
\begin{equation}
u_i\nabla x^p=\nabla(x^pu_i)-\nabla {h}_{pi}-{\bf w}_{pi}=\nabla q_{pi}-{\bf w}_{pi}
\end{equation}
with $\nabla q_{pi}=\nabla(x^pu_i)-\nabla {h}_{pi}$ and $q_{pi}\in  H_{2,D}^2(\Omega)$.
Hence, we derive
\begin{equation}
\| u_i\|^2=\|\nabla q_{pi}\|^2+\|{\bf w}_{pi}\|^2.
\end{equation}

\begin{lemma} For any $p$ and $i$, 
\begin{equation}
3\|\langle\nabla x^p,\nabla u_i\rangle\|^2-2\Lambda_i\|\nabla q_{pi}\|^2
=\dfrac12-\dfrac12\Lambda_i\|u_i\|^2.
\end{equation}
\end{lemma}
\begin{proof} Since, from the Stokes' formula, 
\begin{equation*}
\begin{aligned}
&\int  x^pu_i\langle\nabla x^p, \nabla(\Delta u_i)\rangle\\
&=\int  \Delta(x^pu_i)\langle\nabla x^p, \nabla u_i\rangle\\
&=-\int  \langle u_i\nabla x^p, \nabla\bigl(\Delta (x^pu_i)\bigl)\rangle\\
&=-\int  \langle \nabla q_{pi}, \nabla\bigl(\Delta (x^pu_i)\bigl)\rangle \qquad   (\text{from (2.4) and (2.9)}) \\
&=\int  q_{pi} \Delta^2 (x^pu_i)\\
&=\int  q_{pi} \bigl(4\langle \nabla x^p, \nabla (\Delta u_i)\rangle-\Lambda_ix^p \Delta u_i\bigl)\\
&=-4\int  \Delta u_i\langle \nabla q_{pi}, \nabla x^p\rangle-\Lambda_i\int q_{pi}x^p \Delta u_i\\
\end{aligned}
\end{equation*}
and 
\begin{equation*}
\begin{aligned}
&-\Lambda_i\int q_{pi}x^p \Delta u_i\\
&=\Lambda_i\int \langle \nabla q_{pi}, x^p \nabla u_i\rangle
+\Lambda_i\int q_{pi}\langle \nabla  x^p, \nabla u_i\rangle\\
&=\Lambda_i\int \langle \nabla q_{pi}, x^p \nabla u_i\rangle
-\Lambda_i\int \langle \nabla q_{pi}, u_i\nabla  x^p\rangle\\
&=\Lambda_i\int \langle \nabla q_{pi}, x^p \nabla u_i\rangle
-\Lambda_i\|\nabla q_{pi}\|^2,\\
\end{aligned}
\end{equation*}

\begin{equation*}
\begin{aligned}
&-4\int  \Delta u_i\langle \nabla q_{pi}, \nabla x^p\rangle\\
&=-4\int  \langle \nabla (\Delta q_{pi}), u_i\nabla x^p\rangle\\
&=4\int  \Delta q_{pi}  \langle \nabla x^p, \nabla u_i\rangle\\
&=-4\int    \langle\nabla q_{pi}, \nabla\langle \nabla x^p, \nabla u_i\rangle\rangle\\
&=-4\int    \langle u_i\nabla x^p, \nabla\langle \nabla x^p, \nabla u_i\rangle\rangle\\
&=4\|\langle \nabla x^p, \nabla u_i\rangle\|^2,
\end{aligned}
\end{equation*}
we obtain
\begin{equation}
\begin{aligned}
&\int  x^pu_i\langle\nabla x^p, \nabla(\Delta u_i)\rangle=4\|\langle \nabla x^p, \nabla u_i\rangle\|^2
+\Lambda_i\int \langle \nabla q_{pi}, x^p \nabla u_i\rangle
-\Lambda_i\|\nabla q_{pi}\|^2.\\
\end{aligned}
\end{equation}
From the lemma 2.1 and the above equality, we have 
\begin{equation}
\begin{aligned}
&6\|\langle \nabla x^p, \nabla u_i\rangle\|^2
-2\Lambda_i\|\nabla q_{pi}\|^2-1=-2\Lambda_i\int \langle \nabla q_{pi}, x^p \nabla u_i\rangle.\\
\end{aligned}
\end{equation}
Furthermore, from (2.4), 
$x^p\nabla u_i=\nabla {h}_{pi}+{\bf w}_{pi}$ 
and $\nabla q_{pi}=\nabla(x^pu_i)-\nabla {h}_{pi}$, we have 
\begin{equation}
\begin{aligned}
&\int \langle \nabla q_{pi}, x^p \nabla u_i\rangle\\
&=\int \langle \nabla q_{pi}, \nabla h_{pi}\rangle\\
&=\int \langle \nabla q_{pi}, \nabla (x^pu_i)-\nabla q_{pi}\rangle\\
&=\int \langle \nabla q_{pi}, \nabla (x^pu_i)\rangle-\|\nabla q_{pi}\|^2\\
&=\int \langle u_i\nabla x^{p}, \nabla (x^pu_i)\rangle-\|\nabla q_{pi}\|^2\\
&=\|u_i\|^2+\int \langle u_i\nabla x^{p}, x^p\nabla u_i\rangle-\|\nabla q_{pi}\|^2.\\
\end{aligned}
\end{equation}
Since 
$$
\int \langle u_i\nabla x^{p}, x^p\nabla u_i\rangle=-\|u_i\|^2-\int \langle u_i\nabla x^{p}, x^p\nabla u_i\rangle,
$$
we obtain
$$
\int \langle u_i\nabla x^{p}, x^p\nabla u_i\rangle=-\frac12\|u_i\|^2.
$$
According to (2.13) and (2.14), we have 
\begin{equation*}
3\|\langle\nabla x^p,\nabla u_i\rangle\|^2-2\Lambda_i\|\nabla q_{pi}\|^2
=\dfrac12-\dfrac12\Lambda_i\|u_i\|^2.
\end{equation*}
It finishes the proof of the lemma 2.2.
\end{proof}

\begin{lemma} For any $i$,
\begin{equation}
\Lambda_i \sum_{p=1}^n\| {\bf w}_{pi}\|^2\geq (n-1)
\end{equation}
holds.
\end{lemma}

\begin{proof}
Since
\begin{equation}
\nabla_{\beta}(x^p\nabla_{\alpha} u_i)-
\nabla_{\alpha}(x^p\nabla_{\beta}u_i)=\nabla_\beta w_{pi}^{\alpha}
-\nabla_{\alpha} w_{pi}^{\beta},
\end{equation}
where $w_{pi}^{\alpha}=x^p\nabla_{\alpha}u_i-\nabla_{\alpha}h_{pi}$ denotes the $\alpha$-th component of ${\bf w}_{pi}$, we infer, from ${\text {div}} ({\bf w}_{pi})=0$,
\begin{equation}
\begin{aligned}
\|\nabla {\bf w}_{pi}\|^2=&\sum_{\alpha,\beta=1}^n\|
\nabla_{\alpha}w_{pi}^{\beta}\|^2\\
=&\dfrac 12\sum_{\alpha,\beta=1}^n
\|\nabla_{\beta} w_{pi}^{\alpha}
-\nabla_{\alpha} w_{pi}^{\beta}\|^2+\|{\text {div}} ({\bf w}_{pi})\|^2\quad  \\
=&\dfrac 12\sum_{\alpha,\beta=1}^n \|\nabla_{\beta}(x^p\nabla_{\alpha} u_i)-
\nabla_{\alpha}(x^p\nabla_{\beta}u_i)\|^2\\
=&1-\|\nabla_p u_i\|^2.
\end{aligned}
\end{equation}
Furthermore,  we have
\begin{equation*}
\begin{aligned}
&\Delta  w_{pi}^{\alpha}\\
&=\Delta (x^p\nabla_{\alpha}u_i-\nabla_{\alpha}h_{pi})\\
&=\Delta (x^p\nabla_{\alpha}u_i)-\nabla_{\alpha}\biggl(\text{div}(\nabla h_{pi})\biggl)\\
&=\Delta (x^p\nabla_{\alpha}u_i)-\nabla_{\alpha}\biggl(\text{div}(x^p\nabla u_i)\biggl)\\
&=\nabla_p\nabla_{\alpha}u_i-\nabla_{\alpha}x^p\Delta u_i.
\end{aligned}
\end{equation*}
Thus, we obtain
\begin{equation}
\begin{aligned}
&\Delta  {\bf w}_{pi}=\nabla\langle\nabla x^p, \nabla u_i\rangle-\Delta u_i\nabla x^p.
\end{aligned}
\end{equation}
For any positive constant $\epsilon_i$, we have 
\begin{equation}
\begin{aligned}
&\|\nabla {\bf w}_{pi}\|^2=-\int\langle {\bf w}_{pi}, \Delta {\bf w}_{pi}\rangle\\
&=-\int\langle {\bf w}_{pi}, \nabla\langle\nabla x^p, \nabla u_i\rangle-\Delta u_i\nabla x^p\rangle\\
&\le \dfrac{\epsilon_i}{2}\|{\bf w}_{pi}\|^2+\dfrac{1}{2\epsilon_i}\|  \nabla\langle\nabla x^p, \nabla u_i\rangle-\Delta u_i\nabla x^p\|^2.
\end{aligned}
\end{equation}
Since, from (2.17), 
$$
\sum_{p=1}^n\|\nabla {\bf w}_{pi}\|^2=n-1, \quad  \sum_{p=1}^n 
\|  \nabla\langle\nabla x^p, \nabla u_i\rangle\|^2=\Lambda_i,
$$
by taking sum on $p$ from $1$ to $n$ for (2.19), we have 
\begin{equation*}
(n-1)\leq \dfrac{\epsilon_i}2\sum_{p=1}^n\| {\bf w}_{pi}\|^2+\dfrac{n-1}{2\epsilon_i}\Lambda_i.
\end{equation*}
Putting 
$$
\epsilon_i=\sqrt{\dfrac{(n-1)\Lambda_i}{\sum_{p=1}^n\| {\bf w}_{pi}\|^2}}, 
$$
we obtain
\begin{equation*}
\Lambda_i \sum_{p=1}^n\| {\bf w}_{pi}\|^2\geq (n-1).
\end{equation*}
It completes the proof of the lemma 2.3.
\end{proof}

\vskip 5mm
\noindent
{\it Proof of Theorem 1.1.}   Since $\varphi_{pi}$  is a trial  function, 
 from the Rayleigh-Ritz inequality, we have
\begin{equation}
\Lambda_{k+1} \|\nabla \varphi_{pi}\|^2\leq \int  \varphi_{pi}\Delta^2 \varphi_{pi}
=-\int \nabla \varphi_{pi}\cdot\nabla(\Delta \varphi)_{pi}.
\end{equation}
By making use of the same arguments as in Cheng and Yang \cite{cy2}, we have,
for any $p$ and $i$, 
\begin{equation}
(\Lambda_{k+1}-\Lambda_i) \|\nabla\varphi_{pi}\|^2\le
1+3\|\nabla _p u_i\|^2
-\Lambda_i(\|u_i\|^2-\|{\bf w}_{pi}\|^2)+ \sum_{j=1}^k(\Lambda_i-\Lambda_j)
b_{pij}^2.
\end{equation}

\begin{equation}
1+ 2\sum_{j=1}^k b_{pij}c_{pij}
=-2\int_{\Omega} \langle \nabla\varphi_{pi} ,  \nabla\langle\nabla x^p, \nabla u_i\rangle\rangle,
\end{equation}
where
$$
c_{pij}=\int \langle\nabla\langle\nabla x^p, \nabla u_i\rangle,  \nabla u_j\rangle=-c_{pji}.
$$
Hence, we have, for any positive constant $\delta_i$,
$$
\aligned
&(\Lambda_{k+1}-\Lambda_i)^2(1+2\sum_{j=1}^k b_{pij}c_{pij})\\
&= (\Lambda_{k+1}-\Lambda_i)^2
\int_{\Omega} -2\langle \nabla\varphi_{pi} ,  \nabla\langle\nabla x^p, \nabla u_i\rangle
-\sum_{j=1}^k c_{pij}\nabla u_j\rangle\\
&\le  \delta_i(\Lambda_{k+1}-\Lambda_i)^3 \|\nabla\varphi_{pi}\|^2
+\frac 1 {\delta_i} (\Lambda_{k+1}-\Lambda_i)\left (
\|\nabla\langle\nabla x^p, \nabla u_i\rangle\|^2 -\sum_{j=1}^k c_{pij}^2
\right ).
\endaligned
$$
From (2.21) and $\| u_i\|^2=\|\nabla q_{pi}\|^2+\|{\bf w}_{pi}\|^2$, we obtain
\begin{equation}
\begin{aligned}
&(\Lambda_{k+1}-\Lambda_i)^2(1+2\sum_{j=1}^k b_{pij}c_{pij})\\
&\le \delta_i(\Lambda_{k+1}-\Lambda_i)^2\biggl(1+3\|\nabla _p u_i\|^2
-\Lambda_i\|\nabla q_{pi}\|^2+ \sum_{j=1}^k(\Lambda_i-\Lambda_j)
b_{pij}^2. \biggl)\\
&+\frac 1 {\delta_i}(\Lambda_{k+1}-\Lambda_i)
\left (\|\nabla\langle\nabla x^p, \nabla u_i\rangle\|^2 -\sum_{j=1}^k c_{pij}^2\right ).
\end{aligned}
\end{equation}
By taking sum on $p$ from $1$ to $n$,  we derive 
\begin{equation}
\begin{aligned}
&(\Lambda_{k+1}-\Lambda_i)^2(n+2\sum_{p=1}^n\sum_{j=1}^k b_{pij}c_{pij})\\
&\le \delta_i(\Lambda_{k+1}-\Lambda_i)^2\biggl(n+3
-\Lambda_i\sum_{p=1}^n\|\nabla q_{pi}\|^2+ \sum_{p=1}^n\sum_{j=1}^k(\Lambda_i-\Lambda_j)
b_{pij}^2. \biggl)\\
&+\frac 1 {\delta_i}(\Lambda_{k+1}-\Lambda_i)
\left (\Lambda_i -\sum_{p=1}^n\sum_{j=1}^k c_{pij}^2\right ).
\end{aligned}
\end{equation}
From  the lemma 2.2,  the lemma 2.3 and 
\begin{equation*}
\| u_i\|^2=\|\nabla q_{pi}\|^2+\|{\bf w}_{pi}\|^2,
\end{equation*}
we infer 
$$
\Lambda_i\sum_{p=1}^n\|\nabla q_{pi}\|^2\geq \dfrac53.
$$
Thus, we obtain, for any $i$, 
\begin{equation}
\begin{aligned}
&(\Lambda_{k+1}-\Lambda_i)^2(n+2\sum_{p=1}^n\sum_{j=1}^k b_{pij}c_{pij})\\
&\le \delta_i(\Lambda_{k+1}-\Lambda_i)^2\biggl(n+\dfrac 43+ \sum_{p=1}^n\sum_{j=1}^k(\Lambda_i-\Lambda_j)
b_{pij}^2 \biggl)\\
&+\frac 1 {\delta_i}(\Lambda_{k+1}-\Lambda_i)
\left (\Lambda_i -\sum_{p=1}^n\sum_{j=1}^k c_{pij}^2\right ).
\end{aligned}
\end{equation}
By taking sum for $i$ from $1$ to $k$ and noticing that $b_{pij}$
is symmetric and $c_{pij}$ is antisymmetric on $i, j$, we have
\begin{equation}
\begin{aligned}
&n\sum_{i=1}^k (\Lambda_{k+1}-\Lambda_i)^2
-2\sum_{p=1}^n\sum_{i,j=1}^k (\Lambda_{k+1}-\Lambda_i)(\Lambda_i-\Lambda_j)
b_{pij}c_{pij}\\
&\le (n+\dfrac43)\sum_{i=1}^k\delta_i(\Lambda_{k+1}-\Lambda_i)^2
+\sum_{i=1}^k\frac 1 {\delta_i}(\Lambda_{k+1}-
\Lambda_i)\Lambda_i\\
&-\sum_{p=1}^n\sum_{i,j=1}^k \delta_i(\Lambda_{k+1}-\Lambda_i)(\Lambda_i-\Lambda_j)^2b_{pij}^2
-\sum_{i,j=1}^k\dfrac1{\delta_i}(\Lambda_{k+1}-\Lambda_i)c_{pij}^2\\
&+\sum_{p=1}^n\sum_{i,j=1}^k\delta_i (\Lambda_{k+1}-\Lambda_i)(\Lambda_i-\Lambda_j)^2b_{pij}^2\\
&+\sum_{p=1}^n\sum_{i,j=1}^k \delta_i(\Lambda_{k+1}-\Lambda_i)^2(\Lambda_i-\Lambda_j)b_{pij}^2.
\end{aligned}
\end{equation}
Since, for a non-increasing  monotone  sequence $\{\delta_i\}_{i=1}^k$, 
\begin{equation*}
\begin{aligned}
&\sum_{p=1}^n\sum_{i,j=1}^k\delta_i (\Lambda_{k+1}-\Lambda_i)(\Lambda_i-\Lambda_j)^2b_{pij}^2
+\sum_{p=1}^n\sum_{i,j=1}^k \delta_i(\Lambda_{k+1}-\Lambda_i)^2(\Lambda_i-\Lambda_j)b_{pij}^2\\
&=\dfrac12\sum_{p=1}^n\sum_{i,j=1}^k (\Lambda_{k+1}-\Lambda_i)(\Lambda_{k+1}-\Lambda_j)(\Lambda_i-\Lambda_j)(\delta_i-\delta_j)b_{pij}^2\leq 0.
\end{aligned}
\end{equation*}
We conclude from (2.26) and the above formula, for a non-increasing  monotone 
sequence $\{\delta_i\}_{i=1}^k$,
$$
n\sum_{i=1}^k (\Lambda_{k+1}-\Lambda_i)^2
\le (n+\dfrac43)\sum_{i=1}^k\delta_i(\Lambda_{k+1}-\Lambda_i)^2
+\sum_{i=1}^k\dfrac 1 {\delta_i}  (\Lambda_{k+1}-
\Lambda_i)\Lambda_i.
$$
In particular,  putting 
$$
\delta_i=\dfrac {n}{2(n+\frac43)}
$$
for any $i$,  we obtain
$$
\sum_{i=1}^k (\Lambda_{k+1}-\Lambda_i)^2
\le
\frac {4(n+\frac43)}{n^2} \sum_{i=1}^k
(\Lambda_{k+1}-\Lambda_i)\Lambda_i.
$$
This finishes the proof of the theorem 1.1.
\begin{flushright}
$\Box$
\end{flushright}

\begin{Rem} If
 one  can prove, for any $i$, 
$$
\Lambda_i\sum_{p=1}^n\|\nabla q_{pi}\|^2\geq 3,
$$
one will infer
$$
\sum_{i=1}^k(\Lambda_{k+1}-\Lambda_i)^2\le\frac 4n
\sum_{i=1}^k(\Lambda_{k+1}-\Lambda_i)\Lambda_i,
$$
which solves the conjecture.
\end{Rem}

\section{Proof of the theorem 1.2}

For the unit sphere
$$
S^n(1)=\biggl\{(x^1, x^2, \cdots, x^{n+1}) \in \mathbf{R}^{n+1}; \sum_{i=1}^{n+1}(x^p)^2=1\biggl\},
$$
we denote the induced metric on $S^n(1)$ by the canonical metric $\langle \cdot, \cdot\rangle$ 
on $\mathbf{R}^{n+1}$ also. For any $p$, we have 
\begin{equation}
\nabla_i\nabla_jx^p=-g_{ij}x^p, \qquad \Delta x^p=-nx^p,
\end{equation}
where $g_{ij}$ denotes components of the metric tensor of $S^n(1)$.
Let $u_i$ be  the $i$-th orthonormal eigenfunction of the buckling  problem
(1.2) corresponding to the eigenvalue $\Lambda_i$,
namely,  $ u_i$ satisfies
\begin{equation}
\begin{cases}
&\Delta^2u_i  =-\Lambda_i \Delta u_i \quad in \ \ \Omega ,\\
& u_i |_{\partial \Omega}=\left .\frac {\partial u_i}{\partial \nu}
\right|_{\partial \Omega} =0\\
&(u_i,u_j)_D=\int_{\Omega} \langle\nabla u_i, \nabla u_j \rangle=\delta_{ij}.
\end{cases}
\end{equation}
For constructing trial functions, we use the same notations as in the 
section 2. We would like to  remark that  vector-valued functions in this section 
have  $n+1$ components. Although the orders of differentiations of functions
in the Euclidean space can be exchanged freely, we must do it very carefully for  the 
covariant differentiations of functions in the case of the unit sphere.

Since  $x^p$ for $p=1, 2, \cdots, n+1$ is a coordinate function
of $\mathbf {R}^{n+1}$,
for  the vector-valued function  $x^p \nabla u_i, i=1, \dots, k$, we decompose
it into
\begin{equation}
x^p \nabla u_i= \nabla h_{pi} + {\bf w}_{pi},
\end{equation}
where  $h_{pi} \in H^2_{2,D}(\Omega)$ and $\nabla h_{pi}$ is the projection of $x^p\nabla u_i $
onto ${\bf H}^2_{1,D}(\Omega)$ and
${\bf w}_{pi} \perp H^2_{1,D} (\Omega)$. Thus, 
we have, for any function $h\in C^1(\Omega)\cap L^2(\Omega)$,
\begin{equation}
( {\bf w}_{pi} , \nabla h )=0.
\end{equation}
Hence,  ${\bf w}_{pi}$ satisfies
\begin{equation}
\begin{cases}
&{\bf w}_{pi}|_{\partial \Omega}=0,\\
&\|\text { div}{\bf w}_{pi}\|^2=0.
\end{cases}
\end{equation}
We define function $\varphi_{pi}$ by
\begin{equation}
\varphi_{pi}=h_{pi}-\sum_{j=1}^k b_{pij}u_j,
\end{equation}
where 
$$
b_{pij}=\int x^p\langle\nabla u_i, \nabla u_j\rangle=b_{pji}.
$$ 
It is easy to check  that 
$\varphi_{pi}$ satisfies
\begin{equation*}
\varphi_{pi}|_{\partial \Omega}
=\frac{\partial\varphi_{pi}}{\partial\nu}|_{\partial\Omega}=0\ \
\text{and} \ \ (\varphi_{pi},u_j)_D=(\nabla \varphi_{pi},\nabla u_j)=0,
\end{equation*}
for any $j=1, 2, \cdots, k$, that is, $\varphi_{pi}$ is a trial  function.
Since $\sum_{p=1}^{n+1}(x^p)^2=1$, from (3.3), we have, for any  $i$, 
\begin{equation}
1=  \sum_{p=1}^{n+1} \|\nabla h_{pi}\|^2 +  \sum_{p=1}^{n+1}\|{\bf w}_{pi}\|^2.
\end{equation}
\begin{lemma} For any $i$, we have 
\begin{equation}
 \sum_{p=1}^{n+1}\|{\bf w}_{pi}\|^2 \leq \dfrac{\Lambda_i-(n-1)}{\Lambda_i-(n-2)}.
\end{equation}
\end{lemma}
\begin{proof}
From  $\sum_{p=1}^{n+1}(x^p)^2=1$, we have 
\begin{equation*}
\begin{aligned}
1&=  \sum_{p=1}^{n+1} \|\langle\nabla x^p,\nabla u_i\rangle\|^2\\
&= - \sum_{p=1}^{n+1} \int x^p\text{div}\{\langle\nabla x^p,\nabla u_i\rangle\nabla u_i\}\\
&=- \sum_{p=1}^{n+1} \int x^p\langle\nabla x^p,\nabla u_i\rangle\Delta u_i
- \sum_{p=1}^{n+1} \int \langle x^p\nabla u_i, \nabla\langle\nabla x^p,\nabla u_i\rangle\rangle\\
&=- \sum_{p=1}^{n+1} \int \langle \nabla h_{pi}, \nabla\langle\nabla x^p,\nabla u_i\rangle\rangle.\\
\end{aligned}
\end{equation*}
For any positive constant $\epsilon_i$, we have 
\begin{equation}
1\leq \epsilon_i \sum_{p=1}^{n+1} \|\nabla h_{pi}\|^2 
+ \dfrac 1{4\epsilon_i} \sum_{p=1}^{n+1}\|\nabla \langle \nabla x^p,\nabla u_i\rangle\|^2
\end{equation}
According to the following Bochner formula for a smooth function $f$:
\begin{equation*}
\begin{aligned}
\dfrac12\Delta |\nabla f|^2&
=|\nabla^2f|^2+\langle\nabla f,\nabla (\Delta f)\rangle +\text{Ric}(\nabla f, \nabla f)\\
&=|\nabla^2f|^2+\langle\nabla f,\nabla (\Delta f)\rangle +(n-1)|\nabla f|^2,\\
\end{aligned}
\end{equation*}
where Ric and $\nabla^2f$ denote the Ricci tensor of $S^n(1)$   and the Hessian of $f$,
respectively, we can derive, from (3.1) and by making use of a direct computation, 
\begin{equation}
\Delta  \langle \nabla x^p,\nabla u_i\rangle=
-2x^p\Delta u_i+ \langle \nabla x^p,\nabla (\Delta u_i)\rangle +(n-2) \langle \nabla x^p,\nabla u_i\rangle.
\end{equation} 
Hence, we have 
\begin{equation*}
\begin{aligned}
&\sum_{p=1}^{n+1}\|\nabla \langle \nabla x^p,\nabla u_i\rangle\|^2\\
&=-\sum_{p=1}^{n+1} \int \langle \nabla x^p,\nabla u_i\rangle\Delta  \langle x^p,\nabla u_i\rangle\\
&=-\sum_{p=1}^{n+1} \int \langle \nabla x^p,\nabla u_i\rangle\biggl\{ -2x^p\Delta u_i+ \langle \nabla x^p,\nabla (\Delta u_i)\rangle +(n-2) \langle \nabla x^p,\nabla u_i\rangle \biggl\}\\
&=-\sum_{p=1}^{n+1} \biggl\{\int \langle\nabla x^p,\nabla u_i\rangle\langle \nabla x^p,\nabla (\Delta u_i)\rangle +(n-2) \langle \nabla x^p,\nabla u_i\rangle^2 \biggl\}\\
&=-\int \langle\nabla u_i,\nabla (\Delta u_i)\rangle -(n-2) \|\nabla u_i\|2\\
&=\Lambda_i-(n-2),
\end{aligned}
\end{equation*}
that is,
\begin{equation}
\begin{aligned}
&\sum_{p=1}^{n+1}\|\nabla \langle \nabla x^p,\nabla u_i\rangle\|^2=\Lambda_i-(n-2).
\end{aligned}
\end{equation}
Here we have used
$$
\sum_{p=1}^{n+1} \int \langle\nabla x^p,\nabla u_i\rangle\langle \nabla x^p,\nabla (\Delta u_i)\rangle
=\int \langle\nabla u_i,\nabla (\Delta u_i)\rangle. 
$$
Therefore, from (3.9), we obtain
\begin{equation*}
1\leq \epsilon_i \sum_{p=1}^{n+1} \|\nabla h_{pi}\|^2 
+ \dfrac 1{4\epsilon_i}\biggl(\Lambda_i-(n-2)\biggl)
\end{equation*}
From (3.7), we have 
\begin{equation*}
1+ \epsilon_i \sum_{p=1}^{n+1} \|{\bf w}_{pi}\|^2 
\leq \epsilon_i+ \dfrac 1{4\epsilon_i}\biggl(\Lambda_i-(n-2)\biggl).
\end{equation*}
Taking 
$$
\epsilon_i=\dfrac{\Lambda_i-(n-2)}2,
$$
we complete the proof of the lemma 3.1.
\end{proof}

\vskip 5mm
\noindent
{\it Proof of Theorem 1.2.}
By making use of the trial function $\varphi_{pi}$ and the same argumants
as in Wang and Xia \cite{wx}, we have, for any $p$ and $i$,
\begin{equation}
(\Lambda_{k+1}-\Lambda_i) \|\nabla\varphi_{pi}\|^2\le P_{pi} 
+\|\langle \nabla x^p,\nabla u_i\rangle\|^2
+\Lambda_i\|{\bf w}_{pi}\|^2+ \sum_{j=1}^k(\Lambda_i-\Lambda_j)
b_{pij}^2,
\end{equation}
where
$$
P_{pi}=\int \langle \nabla (x^p)^2,u_i\nabla (\Delta u_i)+\Lambda_iu_i\nabla u_i\rangle.
$$
Defining 
$$
Z_{pi}=\nabla \langle \nabla x^p,\nabla u_i\rangle-\dfrac{n-2}2x^p\nabla u_i, 
$$
$$
c_{pij}=\int \langle \nabla u_j, Z_{pi}\rangle =-c_{pji}
$$
has been proved in Wang and Xia \cite{wx}. Since 
\begin{equation*}
\begin{aligned}
\gamma_{pi}&=-2\int \langle x^p\nabla u_i, Z_{pi}\rangle\\
&=-2\int \langle \nabla h_{pi}+{\bf w}_{pi}, Z_{pi}\rangle\\
&=-2\int \langle \nabla \varphi_{pi}+\sum_{j=1}^kb_{pij}\nabla u_j+{\bf w}_{pi}, Z_{pi}\rangle\\
&=-2\int \langle \nabla \varphi_{pi}, Z_{pi}-\sum_{j=1}^kc_{pij}\nabla u_j\rangle
-2\sum_{j=1}^kb_{pij}c_{pij}+(n-2)\|{\bf w}_{pi}\|^2,
\end{aligned}
\end{equation*}
we have 
\begin{equation*}
\begin{aligned}
\gamma_{pi}+2\sum_{j=1}^kb_{pij}c_{pij}
=-2\int \langle \nabla \varphi_{pi}, Z_{pi}-\sum_{j=1}^kc_{pij}\nabla u_j\rangle
+(n-2)\|{\bf w}_{pi}\|^2.
\end{aligned}
\end{equation*}
Hence, for any positive constant $\delta_i$, we have, according to  (3.12),  
\begin{equation}
\begin{aligned}
&(\Lambda_{k+1}-\Lambda_i)^2\biggl(\gamma_{pi}+2\sum_{j=1}^k b_{pij}c_{pij}\biggl)
-(n-2)(\Lambda_{k+1}-\Lambda_i)^2\|{\bf w}_{pi}\|^2\\
&\le  \delta_i(\Lambda_{k+1}-\Lambda_i)^3 \|\nabla\varphi_{pi}\|^2
+\frac 1 {\delta_i} (\Lambda_{k+1}-\Lambda_i)\left (
\|Z_{pi}\|^2 -\sum_{j=1}^k c_{pij}^2\right )\\
&\leq  \delta_i(\Lambda_{k+1}-\Lambda_i)^2\biggl\{ P_{pi} 
+\|\langle \nabla x^p,\nabla u_i\rangle\|^2
+\Lambda_i\|{\bf w}_{pi}\|^2+ \sum_{j=1}^k(\Lambda_i-\Lambda_j)
b_{pij}^2\biggl\}\\
&+\frac 1 {\delta_i} (\Lambda_{k+1}-\Lambda_i)\left (
\|Z_{pi}\|^2 -\sum_{j=1}^k c_{pij}^2\right ).\\
\end{aligned}
\end{equation}
By taking sum on $p$ from $1$ to $n$,  we derive 
\begin{equation}
\begin{aligned}
&(\Lambda_{k+1}-\Lambda_i)^2\sum_{p=1}^{n+1}\biggl(\gamma_{pi}+2\sum_{j=1}^k b_{pij}c_{pij}\biggl)-(n-2)(\Lambda_{k+1}-\Lambda_i)^2\sum_{p=1}^{n+1}\|{\bf w}_{pi}\|^2\\
&\leq  \delta_i(\Lambda_{k+1}-\Lambda_i)^2\sum_{p=1}^{n+1}\biggl\{ P_{pi} 
+\|\langle \nabla x^p,\nabla u_i\rangle\|^2\\
&+\Lambda_i\|{\bf w}_{pi}\|^2+ \sum_{j=1}^k(\Lambda_i-\Lambda_j)
b_{pij}^2\biggl\}\\
&+\frac 1 {\delta_i} (\Lambda_{k+1}-\Lambda_i)\sum_{p=1}^{n+1}\left (
\|Z_{pi}\|^2 -\sum_{j=1}^k c_{pij}^2\right ).\\
\end{aligned}
\end{equation}
Since 
\begin{equation*}
\begin{aligned}
\gamma_{pi}&=-2\int \langle x^p\nabla u_i, Z_{pi}\rangle\\
&=-2\int \langle x^p\nabla u_i, \nabla \langle \nabla x^p,\nabla u_i\rangle-\dfrac{n-2}2x^p\nabla u_i\rangle\\
&=2\int \langle\nabla  x^p, \nabla u_i\rangle^2+2\int \Delta u_i \langle x^p\nabla x^p,\nabla u_i\rangle
+(n-2)\int (x^p)^2\langle \nabla u_i,\nabla u_i\rangle,\\
\end{aligned}
\end{equation*}
we have 
\begin{equation*}
\begin{aligned}
\sum_{p=1}^{n+1}\gamma_{pi}=n
\end{aligned}
\end{equation*}
From the definition of $Z_{pi}$, we have 
\begin{equation*}
\begin{aligned}
&\sum_{p=1}^{n+1}\|Z_{pi}\|^2\\
&=\sum_{p=1}^{n+1}\int|\nabla \langle \nabla x^p,\nabla u_i\rangle-\dfrac{n-2}2x^p\nabla u_i|^2\\
&=\sum_{p=1}^{n+1}\biggl\{\|\nabla \langle \nabla x^p,\nabla u_i\rangle\|^2
-(n-2)\int \langle \nabla \langle \nabla x^p,\nabla u_i\rangle, x^p\nabla u_i\rangle
+\dfrac{(n-2)^2}4\|x^p\nabla u_i\|^2\biggl\}\\
&=\Lambda_i+\frac{(n-2)^2}4 \ \ (\text{from} \ (3.11)).
\end{aligned}
\end{equation*}
Since $P_{pi}=\int \langle \nabla (x^p)^2,u_i\nabla (\Delta u_i)+\Lambda_iu_i\nabla u_i\rangle$,
we have 
$$
\sum_{p=1}^{n+1}P_{pi}=0.
$$
From the lemma 3.1 and (3.14), we obtain
\begin{equation*}
\begin{aligned}
&(\Lambda_{k+1}-\Lambda_i)^2\biggl(n+2\sum_{p=1}^{n+1}\sum_{j=1}^k b_{pij}c_{pij}\biggl)
-(n-2)(\Lambda_{k+1}-\Lambda_i)^2\dfrac{\Lambda_i-(n-1)}{\Lambda_i-(n-2)}\\
&\leq  \delta_i(\Lambda_{k+1}-\Lambda_i)^2\biggl\{ 1
+\Lambda_i\dfrac{\Lambda_i-(n-1)}{\Lambda_i-(n-2)}+ \sum_{p=1}^{n+1}\sum_{j=1}^k(\Lambda_i-\Lambda_j)
b_{pij}^2\biggl\}\\
&+\frac 1 {\delta_i} (\Lambda_{k+1}-\Lambda_i)\biggl(\Lambda_i+\frac{(n-2)^2}4\biggl)
-\frac 1 {\delta_i} (\Lambda_{k+1}-\Lambda_i)\sum_{p=1}^{n+1}\sum_{j=1}^k c_{pij}^2,\\
\end{aligned}
\end{equation*}
that is,
\begin{equation}
\begin{aligned}
&2(\Lambda_{k+1}-\Lambda_i)^2+(n-2)\dfrac{(\Lambda_{k+1}-\Lambda_i)^2}{\Lambda_i-(n-2)}\\
&\leq  \delta_i(\Lambda_{k+1}-\Lambda_i)^2\biggl\{\Lambda_i-\dfrac{(n-2)}{\Lambda_i-(n-2)}\biggl\}
+\frac 1 {\delta_i} (\Lambda_{k+1}-\Lambda_i)\biggl(\Lambda_i+\frac{(n-2)^2}4\biggl)\\
&-2(\Lambda_{k+1}-\Lambda_i)^2\sum_{p=1}^{n+1}\sum_{j=1}^k b_{pij}c_{pij}
+ \delta_i(\Lambda_{k+1}-\Lambda_i)^2\sum_{p=1}^{n+1}\sum_{j=1}^k(\Lambda_i-\Lambda_j)b_{pij}^2\\
&-\frac 1 {\delta_i} (\Lambda_{k+1}-\Lambda_i)\sum_{p=1}^{n+1}\sum_{j=1}^k c_{pij}^2.\\
\end{aligned}
\end{equation}
Since, for a non-increasing  monotone  sequence $\{\delta_i\}_{i=1}^k$, 
\begin{equation*}
\begin{aligned}
&\sum_{p=1}^n\sum_{i,j=1}^k\delta_i (\Lambda_{k+1}-\Lambda_i)(\Lambda_i-\Lambda_j)^2b_{pij}^2
+\sum_{p=1}^n\sum_{i,j=1}^k \delta_i(\Lambda_{k+1}-\Lambda_i)^2(\Lambda_i-\Lambda_j)b_{pij}^2\\
&=\dfrac12\sum_{p=1}^n\sum_{i,j=1}^k (\Lambda_{k+1}-\Lambda_i)(\Lambda_{k+1}-\Lambda_j)(\Lambda_i-\Lambda_j)(\delta_i-\delta_j)b_{pij}^2\leq 0
\end{aligned}
\end{equation*}
and 
\begin{equation*}
\begin{aligned}
&-2\sum_{i=1}^k(\Lambda_{k+1}-\Lambda_i)^2\sum_{p=1}^{n+1}\sum_{j=1}^k b_{pij}c_{pij}
-\sum_{i=1}^k\delta_i (\Lambda_{k+1}-\Lambda_i)\sum_{p=1}^n\sum_{j=1}^k(\Lambda_i-\Lambda_j)^2b_{pij}^2\\
&-\sum_{i=1}^k\frac 1 {\delta_i} (\Lambda_{k+1}-\Lambda_i)\sum_{p=1}^{n+1}\sum_{j=1}^k c_{pij}^2\\
&=-\sum_{p=1}^n\sum_{i,j=1}^k
\biggl(\sqrt{\delta_i (\Lambda_{k+1}-\Lambda_i)}(\Lambda_i-\Lambda_j)b_{pij}
-\frac 1 {\sqrt{\delta_i}} \sqrt{(\Lambda_{k+1}-\Lambda_i)}c_{pij}\biggl)^2\leq 0,
\\
\end{aligned}
\end{equation*}
by taking sum on $i$ from $1$ to $k$ for (3.15), we obtain
\begin{equation}
\begin{aligned}
&2\sum_{i=1}^k(\Lambda_{k+1}-\Lambda_i)^2+(n-2)\sum_{i=1}^k\dfrac{(\Lambda_{k+1}-\Lambda_i)^2}{\Lambda_i-(n-2)}\\
&\leq \sum_{i=1}^k \delta_i(\Lambda_{k+1}-\Lambda_i)^2
\biggl\{\Lambda_i-\dfrac{(n-2)}{\Lambda_i-(n-2)}\biggl\}\\
&+\sum_{i=1}^k\frac 1 {\delta_i} (\Lambda_{k+1}-\Lambda_i)\biggl(\Lambda_i+\frac{(n-2)^2}4\biggl).\\
\end{aligned}
\end{equation}
It completes the proof of the theorem 1.2.
\begin{flushright}
$\Box$
\end{flushright}

\end{document}